\documentclass{article}

\usepackage[english]{babel}
\usepackage{amssymb}
\usepackage{amsmath}
\usepackage{amsthm}
\usepackage{color}

\newtheorem{theorem}{Theorem}[section]
\newtheorem{definition}[theorem]{Definition}
\newtheorem{lemma}[theorem]{Lemma}
\newtheorem{proposition}[theorem]{Proposition}
\newtheorem{corollary}[theorem]{Corollary}
\newtheorem{remark}[theorem]{Remark}

\newcommand{\hh}{{\mathbb{H}}}

\newcommand{\cc}{{\mathbb{C}}}
\newcommand{\rr}{{\mathbb{R}}}
\newcommand{\nn}{{\mathbb{N}}}
\newcommand{\s}{{\mathbb{S}}}
\newcommand{\z}{{\mathcal{Z}}}
\newcommand{\f}{{\mathcal{F}}}

\newcommand{\mr}{{\mathcal{M}}}

\newcommand{\rft}{{\mathfrak{G}}}

\newcommand{\I}{{\mathrm{Id}}}
\newcommand{\B}{{\mathbb{B}}}

\newcommand{\reg}{\frak{Reg}(\B,\B)}

\title{\bf The Schwarz-Pick lemma for slice regular functions}
\author{Cinzia Bisi\footnote{Partially supported by GNSAGA of the INdAM and by FIRB ``Geometria Differenziale Complessa e Dinamica Olomorfa''.}\\
\normalsize Universit\`a degli Studi di Ferrara\\
\normalsize Dipartimento di Matematica e Informatica\\
\normalsize Via Machiavelli 35, 44121 Ferrara, Italy\\
\normalsize cinzia.bisi@unife.it
\and
Caterina Stoppato$^*$ \footnote{Partially supported by FSE and by Regione Lombardia.}\\ 
\normalsize Universit\`a degli Studi di Milano\\
\normalsize Dipartimento di Matematica ``F. Enriques''\\
\normalsize Via Saldini 50, 20133 Milano, Italy\\  
\normalsize caterina.stoppato@unimi.it\\}

\date{  }

\begin{document}
\maketitle


\begin{abstract}
The celebrated Schwarz-Pick lemma for the complex unit disk is the basis for the study of hyperbolic geometry in one and in several complex variables. In the present paper, we turn our attention to the quaternionic unit ball $\B$. We prove a version of the Schwarz-Pick lemma for self-maps of $\B$ that are slice regular, according to the definition of Gentili and Struppa. The lemma has interesting applications in the fixed-point case, and it generalizes to the case of vanishing higher order derivatives.
\end{abstract}



\section{Introduction} \label{classicM}

In the complex case, holomorphy plays a crucial role in the study of the intrinsic geometry of the unit disc $\Delta = \{z \in \cc: |z|<1\}$ thanks to the Schwarz-Pick lemma \cite{pick2,pick1}.

\begin{theorem}
Let $f:\Delta \to \Delta$ be a holomorphic function and let $z_0 \in \Delta$. Then
\begin{equation}
\left|\frac{f(z)-f(z_0)}{1-\overline{f(z_0)}f(z)}\right| \leq \left|\frac{z-z_0}{1-\bar z_0 z}\right|,
\end{equation}
for all $z \in \Delta$ and
\begin{equation}\label{complexdiff}
\frac{\left|f'(z_0)\right|}{1-|f(z_0)|^2} \leq \frac{1}{1-|z_0|^2}.
\end{equation}
All inequalities are strict for $z \neq z_0$, unless $f$ is a M\"obius transformation of $\Delta$.
\end{theorem}

Some well-known consequences concern the rigidity of holomorphic self-maps of $\Delta$. For instance: 
\begin{corollary}
A holomorphic $f:\Delta \to \Delta$ having more than one fixed point in $\Delta$ must be the identity. 
\end{corollary}

Furthermore, for any holomorphic $f: \Delta \to \Delta$ such that $f(z_0)=z_1$ for fixed $z_0,z_1 \in\Delta$,
the modulus $|f'(z_0)|$ cannot exceed $\frac{1-|z_1|^2}{1-|z_0|^2}$ and it reaches this value if and only if it is a M\"obius transformation, \cite{vesentini}. This implies the following special case of the Cartan-Carath\'eodory theorem.
\begin{theorem}
Let $f$ be a holomorphic self-map of $\Delta$. If $z_0$ is a fixed point of $f$ and $f'(z_0) = 1$ then $f$ coincides with the identity function.
\end{theorem}

These are the bases for the study of hyperbolic geometry in one and in several complex variables. We refer to \cite{abate,vesentini} for the foundations of this beautiful theory.

Versions of the Schwarz lemma have been proven for the open unit ball 
$$\mathbb{B}=\{q \in \hh\ |\  |q|<1\}$$
of the real space of quaternions $\hh$.  
Within the theory of regular quaternionic functions introduced by Fueter in \cite{fueter1,fueter2}, which has long been the most successful analog of holomorphy over the quaternions, the article \cite{yang} presents a version of the Schwarz lemma for functions $\hh \setminus \overline{\B} \to \B$ that map $\infty$ to $0$ and that are Fueter-regular, i.e., that lie in the kernel of $\frac{\partial}{\partial x_0}+i\frac{\partial}{\partial x_1}+ j\frac{\partial}{\partial x_2}+k\frac{\partial}{\partial x_3}$. (More generally, the analog of the Schwarz lemma presented in \cite{yang} is concerned with functions over the Clifford Algebras $Cl(0,m)$). See \cite{librosommen,sudbery} for the foundations of Fueter's theory and of its generalization to the Clifford setting.

Another theory of quaternionic functions, introduced in \cite{cras,advances}, is based on a different notion of regularity.

\begin{definition}
Let $\Omega$ be a domain in $\hh$ and let $f : \Omega \to \hh$ be a function. For all $I \in \s = \{q \in \hh \ |\  q^2 = -1\}$, let us denote $L_I = \rr + I \rr$, $\Omega_I = \Omega \cap L_I$ and $f_I = f_{|_{\Omega_I}}$. 
The function $f$ is called (Cullen or) \emph{slice regular} if, for all $I \in \s$, the restriction $f_I$ is real differentiable and the function $\bar \partial_I f : \Omega_I \to \hh$ defined by
$$
\bar \partial_I f (x+Iy) = \frac{1}{2} \left( \frac{\partial}{\partial x}+I\frac{\partial}{\partial y} \right) f_I (x+Iy)
$$
vanishes identically.
\end{definition}

The same articles introduce the \emph{Cullen derivative} $\partial_cf$ of a slice regular function $f$ as
\begin{equation}\label{cullen}
\partial_cf(x+Iy)=\frac{1}{2}\left(\frac{\partial}{\partial x}-I\frac{\partial}{\partial y}\right)f(x+Iy)
\end{equation}
for $I \in \s,\ x,y \in \rr$, and they present an analog of the Schwarz lemma.

\begin{theorem}
Let $f : \B \to \B$ be a slice regular function. If $f(0) = 0$ then 
\begin{equation}
|f(q)|\leq |q|
\end{equation}
for all $q \in \B$ and 
\begin{equation}\label{R-schwarzeq}
|\partial_c f(0)|\leq 1. 
\end{equation}
Both inequalities are strict (except at $q=0$) unless $f(q) = q u$ for some $u \in \partial \B = \{q \in \hh\ |\ |q|=1\}$.
\end{theorem}

We are presently interested in recovering the full Schwarz-Pick lemma for $\B$.
It is known in literature that the set $\mathbb{M}$ of \emph{(classical) M\"obius transformations} of $\B$,
\begin{equation}\label{InvM}
\mathbb{M}=\left\{g(q) = v(q-q_0) (1-\bar q_0q)^{-1} u: u,v \in \partial \B, q_0 \in \B\right\},
\end{equation}
is a group with respect to the composition operation and that it is isomorphic to $Sp(1,1)/\{\pm \I\}$. We recall that
$$Sp(1,1) = \{C \in GL(2, \hh) \ |\  \overline C ^t H C = H\}\leq SL(2, \hh),$$
where $H = 
\begin{bmatrix}
1 & 0 \\
0 & -1
\end{bmatrix}$, $GL(2,\hh)$ denotes the group of $2 \times 2$ invertible quaternionic matrices, and $SL(2,\hh)$ denotes the subgroup of those such matrices which have unit Dieudonn\'e determinant (for details, see \cite{poincare} and references therein). Among the works that treat this matter, even in the more general context of Clifford Algebras, let us mention \cite{ahlforsmoebius, maass, vahlen}. 

The group $\mathbb{M}$, and more in general the group of classical linear fractional transformations $q \mapsto (aq+b) (cq+d)^{-1}$, is not included in Fueter's class. The identity function and the rotations $q \mapsto vq$ and $q \mapsto qu$ with $u,v \in \partial \B$ are examples of classical M\"obius transformations that are not Fueter-regular. Thanks to a result of \cite{sudbery}, and following \cite{perotti2009}, one can associate to each transformation $g(q) =   v(q-q_0) (1-\bar q_0q)^{-1} u$ in $\mathbb{M}$ the Fueter-regular function
$$G(q) = \frac{(1-\bar q_0 q)^{-1}}{|1-\bar q_0 q|^2}\,  u\,\gamma(g(q)),$$
where $\gamma$ is the Fueter-regular function 
$$\gamma(x_0+ix_1+jx_2+kx_3)=x_0+ix_1+jx_2-kx_3.$$
However, the function $G$ is not, in general, a self-map of $\B$.
The variant of the Fueter class considered in \cite{perotti2009}, defined as the kernel of $\frac{\partial}{\partial x_0}+i\frac{\partial}{\partial x_1}+ j\frac{\partial}{\partial x_2}-k\frac{\partial}{\partial x_3}$, includes the rotations $q \mapsto v q u$ for all $u \in \partial \B$ and for every $v \in \partial \B$ that is reduced, i.e., whose component along $k$ vanishes. However, the treatment of the rest of the classical M\"obius transformations encounters the same kind of difficulties as in Fueter's case. 

On the other hand, the class of slice regular functions includes the transformations $q \mapsto (q-q_0)(1-q_0q)^{-1}u$ for $u \in \partial \B$ and $q_0$ in the real interval $(-1,1)$. It does not contain the whole group $\mathbb{M}$, but \cite{moebius} introduced the new class of \emph{(slice) regular M\"obius transformations} of $\B$, which are nicely related to the classical ones. They are presented in detail in section \ref{regularmoebius}, which also illustrates several operations that preserve slice regularity: the multiplication $f(q)*g(q)$ of $f(q)$ and $g(q)$, with respect to which every $g \not \equiv 0$ admits an inverse $g(q)^{-*}$; the conjugation $f^c(q)$; and the symmetrization $f^s(q) = f(q)*f^c(q)$.

In section \ref{schwarzpick}, we prove a quaternionic analog of the Schwarz-Pick lemma, which discloses the possibility of using slice regular functions in the study of the intrinsic geometry of $\B$. Before stating our main result, let us recall  the basic notions concerning the real differential of a slice regular function. At a real point $x_0$, it acts by right multiplication by the Cullen derivative $\partial_cf(x_0)$, while at a point $q_0 = x_0+Iy_0 \in L_I$ with $I \in \s, x_0,y_0 \in \rr$ and $y_0 \neq 0$ it has been thus characterized in \cite{expansion}: if we split the tangent space $T_{q_0}\Omega\cong \hh = \rr^4$ as $L_I \oplus L_I^\perp$ (with respect to the standard scalar product), then the differential of $f$ at $q_0$ acts on $L_I$ by right multiplication by $\partial_cf(q_0)$; on $L_I^\perp$, it acts by right multiplication by the \emph{spherical derivative}
\begin{equation}
\partial_sf(q_0) = (2Im(q_0))^{-1}(f(q_0)-f(\bar q_0))
\end{equation}
defined in \cite{perotti}. We recall that the Cullen derivative is a slice regular function, while the spherical derivative is slice regular only when it is constant. The quaternions $\partial_cf(q_0)$ and $\partial_sf(q_0)$ can be computed as the values at $q_0$ and $\bar q_0$ of a unique slice regular function, which we may call the \emph{differential quotient} of $f$ at $q_0$:

\begin{remark}
Let $f$ be a slice regular function on $B(0,R) = \{q \in \hh \ |\  |q| <R\}$. If, for all $q_0 \in \Omega$, we denote as
\begin{equation}\label{R}
R_{q_0}f(q) = (q-q_0)^{-*}*(f(q)-f(q_0))
\end{equation}
then $\partial_cf(q_0) = R_{q_0}f(q_0)$ and $\partial_sf(q_0)= R_{q_0}f(\bar q_0)$.
\end{remark}

We are now in a position to state the main result of the present article.

\begin{theorem}[Schwarz-Pick lemma] 
Let $f : \B \to \B$ be a regular function and let $q_0 \in \B$. Then in $\B$
\begin{eqnarray}
|(f(q)-f(q_0))*(1-\overline{f(q_0)}*f(q))^{-*}| &\leq& |(q-q_0)*(1-\bar q_0*q)^{-*}| \\
|R_{q_0}f(q)*(1-\overline{f(q_0)}*f(q))^{-*}| &\leq& |(1-\bar q_0*q)^{-*}| 
\end{eqnarray}
Moreover,
\begin{eqnarray}
|\partial_c f *(1-\overline{f(q_0)}*f(q))^{-*}|_{|_{q_0}} &\leq& \frac{1}{1-|q_0|^2}\\
\frac{|\partial_s f(q_0)|}{|1-f^s(q_0)|} &\leq& \frac{1}{|1-\overline{q_0}^2|}
\end{eqnarray}
If $f$ is a slice regular M\"obius transformation of $\B$ then equality holds in the previous formulas. Else, all the aforementioned inequalities are strict (except for the first one at $q_0$, which reduces to $0\leq0$).
\end{theorem}

We conclude section \ref{schwarzpick} computing a point $\tilde q_0$ with $Re(\tilde q_0) = Re(q_0)$ and $|Im(\tilde q_0)| = |Im(q_0)|$ such that
$$|\partial_c f *(1-\overline{f(q_0)}*f(q))^{-*}|_{|_{q_0}} = \frac{|\partial_c f(q_0)|}{|1-\overline{f(q_0)}f(\tilde q_0)|}.$$

As an application of the main theorem, in section \ref{sectionapplications} we obtain direct proofs of the quaternionic analogs of the Cartan Rigidity theorems mentioned at the beginning of this introduction. Versions of these results have been proven in \cite{rigidity}, and our new approach allows to strengthen their statements.

\begin{theorem}
Let $f : \B \to \B$ be a slice regular function and suppose $f$ to have a fixed point $q_0 \in \B$. Then either $f$ is the identity function, or $f$ has no other fixed point in $\B$.
\end{theorem}

\begin{theorem}
Let $f : \B \to \B$ be a slice regular function and suppose $f$ to have a fixed point $q_0 \in \B$. The following facts are equivalent:
\begin{enumerate}
\item[$(1)$] $f$ coincides with the identity function;
\item[$(2)$] the real differential of $f$ at $q_0$ is the identity;
\item[$(3)$] the Cullen derivative $\partial_c f (q_0)$ equals $1$;
\item[$(4)$] the spherical derivative $\partial_s f(q_0)$ equals $1$;
\item[$(5)$] $R_{q_0}f(q)$ equals $(1-\bar q_0 * q)^{-*}*(1-\bar q_0*f(q))$ at some $q \in \B$.
\end{enumerate}
\end{theorem}

Finally, in section \ref{sectionhigher} we generalize our version of the Schwarz-Pick lemma to the case of vanishing higher order derivatives.

\subsection*{Acknowledgements} 
The authors wish to thank Graziano Gentili for his enthusiasm for this topic, and for the fruitful discussions that originated this work. They are also grateful to the anonymous referee for her/his precious suggestions.


\section{Regular M\"obius transformations of $\B$}\label{regularmoebius}

This section surveys the algebraic structure of slice regular functions, and its application to the construction of regular fractional transformations. From now on, we will omit the term `slice' and refer to these functions as regular, {\it tout court}.
Since we will be interested only in regular functions on Euclidean balls $B(0,R)$ of radius $R$ centered at $0$,
or on the whole space $\hh = B(0,+\infty)$, we will follow the presentation of \cite{zeros,poli}. However, we point out that many of the results we are about to mention have been generalized to a larger class of domains in \cite{advancesrevised}.

\begin{theorem} 
Fix $R$ with $0<R\leq + \infty$ and let 
$$\mathcal{D}_R = \{f:B(0,R)\to \hh\ |\ f \mathrm{\ regular}\}.$$
Then $\mathcal{D}_R$  coincides with the set of quaternionic power series $f(q) =\sum_{n \in \nn} q^n a_n$ (with $a_n \in \hh$) converging in $B(0,R)$. Moreover, $\mathcal{D}_R$ is an associative real algebra with respect to $+$ and to the \emph{regular multiplication} $*$ defined on $f(q) =\sum_{n \in \nn} q^n a_n$ and $g(q) =\sum_{n \in \nn} q^n b_n$ by the formula
\begin{equation}
f*g(q)= \sum_{n \in \nn} q^n \sum_{k=0}^n a_k b_{n-k}.\end{equation}
\end{theorem}

We will also write $f(q)*g(q)$ for $f*g(q)$. In this case, the letter $q$ will always denote the variable. 
The ring $\mathcal{D}_R$ admits a classical ring of quotients 
$$\mathcal{L}_R = \{f^{-*}*g \ |\  f,g \in \mathcal{D}_R, f \not \equiv 0\}.$$ 
In order to introduce it, we begin with the following definition.

\begin{definition}\label{conjugate}
Let $f(q) = \sum_{n \in \nn} q^n a_n$ be a regular function on an open ball $B = B(0,R)$. The \textnormal{regular conjugate} of $f$, $f^c : B \to \hh$, is defined as $f^c(q) = \sum_{n \in \nn} q^n \bar a_n$ and the \textnormal{symmetrization} of $f$, as $f^s = f * f^c = f^c*f$. 
\end{definition}

Notice that $f^s(q) = \sum_{n \in \nn} q^n r_n$ with $r_n = \sum_{k = 0}^n a_k \bar a_{n-k} \in \rr$. Moreover, the zero-sets of $f^c$ and $f^s$ have been fully characterized.

\begin{theorem}\label{conjugatezeros}
Let $f$ be a regular function on $B = B(0,R)$. For all $x,y \in \rr$ with $x+y\s \subseteq B$, the regular conjugate $f^c$ has as many zeros as $f$ in $x+y\s$. Moreover, the zero set of the symmetrization $f^s$ is the union of all the $x+y\s$ on which $f$ has a zero.
\end{theorem}

We are now ready for the definition of regular quotient. We denote by 
$$\z_h = \{q \in B \ |\  h(q) = 0\}$$ 
the zero-set of a function $h$.

\begin{definition}\label{quotient}
Let $f,g : B = B(0,R) \to \hh$ be regular functions. The \emph{left regular quotient} of $f$ and $g$ is the function $f^{-*} * g$ defined in $B \setminus \z_{f^s}$ by 
\begin{equation}
f^{-*} * g (q) = {f^s(q)}^{-1} f^c * g(q).
\end{equation} 
Moreover, the \emph{regular reciprocal} of $f$ is the function $f^{-*} = f^{-*} * 1$.
\end{definition}

Left regular quotients proved to be regular in their domains of definition. 
%
If  we set $(f^{-*}*g)*(h^{-*}*k) = (f^{s}h^{s})^{-1} f^c*g*h^c*k$ then $(\mathcal{L}_R,+,*)$ is a division algebra over $\mathbb{R}$ and it is the classical ring of quotients of $(\mathcal{D}_R,+,*)$ (for this notion, see \cite{rowen}). In particular, $\mathcal{L}_R$ coincides with the set of \emph{right regular quotients} 
$$g*h^{-*} (q) = {h^s(q)}^{-1} g * h^c(q).$$ 
The definition of regular conjugation and symmetrization is extended to $\mathcal{L}_R$ setting $(f^{-*}*g)^c = g^c*(f^c)^{-*}$ and $(f^{-*}*g)^s(q) = {f^s(q)}^{-1}g^s(q)$.
Furthermore, the following relation between the left regular quotient $f^{-*} * g(q)$ and the quotient $f(q)^{-1} g (q)$ holds. 

\begin{theorem}\label{quotients}
Let $f,g$ be regular functions on $B=B(0,R)$. Then
\begin{equation}
f*g(q) =
\left\{
\begin{array}{ll}
0 & \mathrm{if\ } f(q)=0\\
 f(q)\, g(f(q)^{-1}qf(q)) &  \mathrm{otherwise}
\end{array}
\right.
\end{equation}
and setting $T_f(q) = f^c(q)^{-1} q f^c(q)$ for all $q \in B \setminus \z_{f^s}$,
\begin{equation}
f^{-*}*g(q) = f(T_f(q))^{-1} g(T_f(q)),
\end{equation}
for all $q \in B \setminus \z_{f^s}$. For all $x,y \in \rr$ with $x+y\s\subset B \setminus \z_{f^s}\subset B \setminus \z_{f^c}$, the function $T_f$ maps $x+y\s$ to itself (in particular $T_f(x) = x$ for all $x \in \rr$). Furthermore, $T_f$ is a diffeomorphism from $B \setminus \z_{f^s}$ onto itself, with inverse $T_{f^c}$.
\end{theorem}

We point out that, so far, no similar result relating $g*h^{-*}(q)$ to $g(q) h(q)^{-1}$ is known.

This machinery allowed the introduction in \cite{moebius} of regular analogs of linear fractional transformations. To each 
$A=\begin{bmatrix}
a & c \\
b & d
\end{bmatrix} \in GL(2, \hh)$ we can associate the \emph{regular fractional transformation} 
$$\f_A (q) = (qc+d)^{-*}*(qa+b).$$
By the formula $(qc+d)^{-*}*(qa+b)$ we denote the aforementioned left regular quotient $f^{-*}*g$ of $f(q) = qc+d$ and $g(q) = qa +b$. We denote the $2 \times 2$ identity matrix as $\I$. The set of regular fractional transformations
$$\rft = \{\f_A \ |\  A \in GL(2,\rr)\}$$
is not a group, but it is the orbit of the identity function $id = \f_\I$ with respect to the two actions on $\mathcal{L}_\infty$ described in the next theorem. 

\begin{theorem}
Choose $R>0$ and consider the ring of quotients of regular quaternionic functions in $B(0,R)$, denoted by $\mathcal{L}_R$.
Setting 
\begin{equation}
f.A = (f c + d)^{-*}*(f a+ b)
\end{equation} 
for all $f \in \mathcal{L}_R$ and for all $A =
\begin{bmatrix}
a & c \\
b & d
\end{bmatrix} \in GL(2,\hh)$,
defines a right action of $GL(2,\hh)$ on $\mathcal{L}_R$. A left action of $GL(2,\hh)$ on $\mathcal{L}_R$ is defined setting
\begin{equation}
A^{t}.f=(a*f+b)*(c*f+d)^{-*}.
\end{equation} 
The stabilizer of any element of $\mathcal{L}_R$ with respect to either action includes the normal subgroup $N=\left \{t \cdot \I\ |\  t \in \rr \setminus \{0\} \right \}\, \unlhd\, GL(2,\hh)$. Both actions are faithful, but not free, when reduced to $PSL(2,\hh) = GL(2,\hh)/N$.
\end{theorem}

For more details, see \cite{volumeindam,moebius}. The two actions are related as follows.

\begin{proposition}\label{hermitian}\label{leftright}
For all $A \in GL(2,\hh)$ and for all $f\in \mathcal{L}_R$
\begin{enumerate}
\item $\left(f.A\right)^c=\bar{A}^t.f^c$;
\item if $A$ is Hermitian then $f.A = A^{t}.f$;
\item if $A$ is Hermitian then  $(f.A)^c= f^c.\bar{A}$.
\end{enumerate}
\end{proposition}

As a consequence, the set $\rft$ of regular fractional transformations is preserved by regular conjugation. For the proofs of these properties, we refer the reader to \cite{volumeindam}.


In the present paper, we are specifically interested in those regular fractional transformations that map the open quaternionic unit ball $\B$ onto itself, called \emph{regular M\"obius transformations  of $\B$}, whose class we denote as
$$\mathfrak{M}= \{f \in \rft \ |\  f(\B) = \B\}.$$
More generally, we will concern ourselves with the class
$$\reg = \{f:\B \to \B\ |\ f \mathrm{\ is\ regular}\}$$
of regular self-maps of $\B$. It was proven in \cite{moebius} that a function $f \in \reg$ is a regular M\"obius transformation if, and only if, it is bijective.  Furthermore, the next property was proven in \cite{volumeindam,moebius}.

\begin{theorem}
A function $f : \B \to \hh$ is a regular M\"obius transformation of $\B$ if and only if there exist (unique) $u \in \partial \B, a \in \B$ such that
\begin{equation}
f(q) = (q-q_0)*(1-\bar q_0* q)^{-*}u= (1-q \bar q_0)^{-*}*(q-q_0)u 
\end{equation}
In other words, $\mathfrak{M}$ is the orbit of the identity function under the left and right actions of $Sp(1,1)$.
\end{theorem}

We point out that, by definition, $\bar q_0* q = q \bar q_0$.
Finally, let us recall a result that will prove useful in the sequel (see proposition 3.3 of \cite{volumeindam}).

\begin{proposition}\label{transform}
If $f \in \reg$ then for all $a \in \B$
\begin{equation}
\left(f(q)-a\right)*\left(1-\bar a*f(q)\right)^{-*}=\left(1-f(q) \bar a\right)^{-*}*\left(f(q)-a\right).
\end{equation}
Furthermore, the left and right actions of $Sp(1,1)$ and the regular conjugation preserve both $\reg$ and $\mathfrak{M}$.
\end{proposition}


\section{The Schwarz-Pick lemma}\label{schwarzpick}

In this section, we shall prove the announced Schwarz-Pick lemma for quaternionic regular functions.
In order to obtain it, we begin with a result concerning the special case of a function $f: \B \to \B$ having a zero. We follow the line of the complex proof, making use of the maximum modulus principle for regular functions proven in \cite{advances}.

\begin{theorem}\label{maximum}
Let $f : B(0,R) \to \hh$ be a regular function. If $|f|$ has a relative maximum at $p \in B(0,R)$, then $f$ is constant.
\end{theorem}

We now turn to the aforementioned result.

\begin{theorem} \label{teo1}
If $f : \B \to \B$ is regular and if $f(q_0)=0$ for some $q_0 \in \B$, and if 
\begin{equation}
\mr_{q_0}(q) = (q-q_0)*(1-q\bar q_0)^{-*}= (1-q\bar q_0)^{-*}*(q-q_0),
\end{equation}
then
\begin{equation}
|\mr_{q_0}^{-*}*f(q)| \leq 1
\end{equation}
for all $q \in \B$. The inequality is strict, unless $\mr_{q_0}^{-*}*f(q) \equiv u$ for some $u \in \partial \B$.
\end{theorem}
\begin{proof}
Let us consider $\mr_{q_0}^{-*}(q) = (1-q\bar q_0)*(q-q_0)^{-*} = (q-q_0)^{-*} *(1-q\bar q_0)$, which is a regular function on $\B$ minus the $2$-sphere $S_{q_0} = x_0 +y_0 \s$ through $q_0$ (that is, minus the zero set of $(q-q_0)^s = (q-x_0)^2+y_0^2$). Since $f(q_0)=0$, we have $f(q) = (q-q_0)*R_{q_0}f(q)$ where $R_{q_0}f:\B \to \hh$ is the differential quotient defined by formula \eqref{R}. Hence setting
$$h(q) = \mr_{q_0}^{-*}*f(q) = (1-q\bar q_0)*R_{q_0}f(q)$$
defines a regular function on $\B$. Moreover, by the first part of theorem \ref{quotients}
$$h(q)=\mr_{q_0}^{-*}*f(q) = \mr_{q_0}^{-*}(q) f(g(q)^{-1}q g(q))$$
where $g = \mr_{q_0}^{-*}$. Since $|f|<1$ in $\B$, we conclude that
$$|h(q)| = \left|\mr_{q_0}^{-*}*f(q)\right| \leq \left|\mr_{q_0}^{-*}(q)\right|$$
away from $S_{q_0}$. 
Applying the second part of theorem \ref{quotients}, we notice that for all $q \in \B \setminus S_{q_0}$
$$\mr_{q_0}^{-*}(q)=(T_l(q)-q_0)^{-1} (1-T_l(q)\bar q_0) =\left[M_{q_0}(T_l(q))\right]^{-1}$$
where $l(q) = q-q_0$, and where $M_{q_0}(q) = (1-q\bar q_0)^{-1}(q-q_0)$. Now, $M_{q_0}$ maps $\B$ onto itself and $\partial \B$ onto itself, and for all $\varepsilon >0$ there exists $r$ with $|q_0|<r<1$ such that
$$1 \leq |M_{q_0}(q)|^{-1}\leq 1+\varepsilon$$
for $|q|\geq r$. Hence,
$$\max_{|q|=r} |h(q)| \leq \max_{|q|=r} |\mr_{q_0}^{-*}(q)| = \max_{|q|=r} |M_{q_0}(T_l(q))|^{-1} = \max_{|w|=r} |M_{q_0}(w)|^{-1} \leq 1+\varepsilon.$$
Suppose there existed $p \in \B, \delta >0$ such that $|h(p)| = 1 + \delta$. There would exist $r>|p|$ (with $r>|q_0|$) such that $\max_{|q|=r} |h(q)| \leq 1+ \delta/2$ and, by the maximum modulus principle $|h(q)| \leq 1+ \delta/2$ for $|q| \leq r$. We would then have $|h(p)| \leq 1+ \delta/2$, a contradiction with the hypothesis. Hence $|\mr_{q_0}^{-*}*f(q)| \leq 1$  for all $q \in \B$.

We conclude observing that, since $|\mr_{q_0}^{-*}*f(q)| \leq 1$ for all $q \in \B$, if there exists $\tilde q \in \B$ such that $|\mr_{q_0}^{-*}*f(\tilde q)| = 1$ then by the maximum modulus principle \ref{maximum}, $\mr_{q_0}^{-*}*f$ must be a constant $u \in \partial \B$.
\end{proof}

In order to reformulate the previous result as an analog of the Schwarz-Pick lemma, we will need some other instruments. The first of them is the following lemma.

\begin{lemma}\label{modulusproduct}
Let $f,g,h: B=B(0,R) \to \hh$ be regular functions. If $|f|\leq |g|$ then $|h*f| \leq |h*g|$. Moreover, if $|f|< |g|$ then $|h*f| < |h*g|$ in $B \setminus \z_h$.
\end{lemma}

\begin{proof}
If $|f|\leq |g|$ then for all $q \in B \setminus \z_h$
$$|f(h(q)^{-1}qh(q))| \leq |g(h(q)^{-1}qh(q))|$$
so that
$$|h*f(q)| = |h(q)|\cdot|f(h(q)^{-1}qh(q))| \leq |h(q)|\cdot|g(h(q)^{-1}qh(q))|| =  |h*g(q)|$$
thanks to theorem \ref{quotients}. The reasoning is also valid if all the inequalities are substituted by strict inequalities. Finally, for all $q \in \z_h$ we have $|h*f(q)|=0=|h*g(q)|$.
\end{proof}

Secondly, let us compute the differential quotient (and the derivatives) of $\mr_{q_0}$ at $q_0$.

\begin{remark}
In the case of the regular M\"obius transformation $\mr_{q_0}$, clearly $R_{q_0}\mr_{q_0}(q) = (1-q\bar q_0)^{-*}$, so that $\partial_c f (q_0) = \frac{1}{1-|q_0|^2}$ and $\partial_s f(q_0)= \frac{1}{1-\overline{q_0}^2}$.
\end{remark}

We are now in a position to suitably restate our result.

\begin{corollary}\label{schwarzp}
If $f : \B \to \B$ is regular, if $q_0 \in \B$ and if $f(q_0)=0$ then
\begin{equation}
|f(q)| \leq |\mr_{q_0}(q)|
\end{equation}
for all $q \in \B$. The inequality is strict at all $q \in \B \setminus \{q_0\}$, unless there exists $u \in \partial \B$ such that $f(q) = \mr_{q_0}(q) \cdot u$ at all $q \in \B$.
Moreover, $|R_{q_0}f(q)| \leq |(1-q\bar q_0)^{-*}|$ in $\B$ and in particular
\begin{eqnarray}
|\partial_c f (q_0)| \le \frac{1}{1-|q_0|^2}\\
|\partial_s f(q_0)|
\leq \frac{1}{|1-\overline{q_0}^2|}.
\end{eqnarray}
These inequalities are strict, unless $f(q) = \mr_{q_0}(q) \cdot u$ for some $u \in \partial \B$.
\end{corollary}

\begin{proof}
By theorem \ref{teo1},
$$|\mr_{q_0}^{-*}*f(q)|\leq 1$$
for all $q \in \B$. Since $\mr_{q_0}^{-*}*f(q)$ and the constant $1$ are regular in $\B$, by lemma \ref{modulusproduct}
$$|f| \leq |\mr_{q_0}|$$
in $\B$. According to theorem \ref{teo1}, all inequalities above are strict for $q \in \B \setminus \{q_0\}$, unless there exists $u \in \partial \B$ such that $\mr_{q_0}^{-*}*f(q) \equiv u$, that is, $f(q) = \mr_{q_0}(q) \cdot u$ for all $q \in \B$.

The second statement follows from
$$1 \geq |\mr_{q_0}^{-*}*f(q)| = |(1-q\bar q_0)*(q-q_0)^{-*}*f(q)|=|(1-q\bar q_0)* R_{q_0}f(q)|$$
applying lemma \ref{modulusproduct}, since $(1-q\bar q_0)^{-*}$ is regular and has no zeros in $\B$.
\end{proof}

We shall now generalize the previous corollary to an analog of the Schwarz-Pick lemma. We will make use of the Leibniz rules for $\partial_c, \partial_s$, from \cite{powerseries} and \cite{perotti} respectively. We recall that if $f$ is regular in $B(0,R)$ then
$$f(q) = v_s f(q) + Im(q) \partial_s f(q)$$
for all $q \in B(0,R)$ where $v_s f$ denotes the \emph{spherical value} $v_s f(q) = \frac{f(q)+f(\bar q)}{2}$ for all $q \in B(0,R)$.

\begin{remark}
If $f:B(0,R) \to \hh$ is a regular function then
\begin{eqnarray}
\partial_c (f*g)(q) &=& \partial_c f(q) * g(q) + f (q)*\partial_c g(q)\label{leibnizc}\\
\partial_s (f*g)(q) &=& \partial_s f(q) \cdot v_sg (q)+ v_s f(q) \cdot \partial_s g(q). \label{leibnizs}
\end{eqnarray}
\end{remark}

\begin{theorem}[Schwarz-Pick lemma] \label{MainSchwarzPick}
Let $f : \B \to \B$ be a regular function and let $q_0 \in \B$. Then in $\B$
\begin{eqnarray}
|(f(q)-f(q_0))*(1-\overline{f(q_0)}*f(q))^{-*}| &\leq& |(q-q_0)*(1-\bar q_0*q)^{-*}| \label{MainSchwarzPickEq}\\
|R_{q_0}f(q)*(1-\overline{f(q_0)}*f(q))^{-*}| &\leq& |(1-\bar q_0*q)^{-*}| \label{Req}
\end{eqnarray}
Moreover,
\begin{eqnarray}
|\partial_c f *(1-\overline{f(q_0)}*f(q))^{-*}|_{|_{q_0}} &\leq& \frac{1}{1-|q_0|^2} \label{culderiv}\\
\frac{|\partial_s f(q_0)|}{|1-f^s(q_0)|} &\leq& \frac{1}{|1-\overline{q_0}^2|}\label{derivsfera}
\end{eqnarray}
If $f$ is a regular M\"obius transformation then equality holds in \eqref{MainSchwarzPickEq},\eqref{Req} for all $q \in \B$, and in \eqref{culderiv},\eqref{derivsfera}. Else, all the aforementioned inequalities are strict (except for \eqref{MainSchwarzPickEq} at $q_0$, which reduces to $0\leq0$).
\end{theorem}

\begin{proof}
Thanks to proposition \ref{transform}, 
$$\tilde f(q) = (f(q)-f(q_0))*(1-\overline{f(q_0)}*f(q))^{-*}$$ 
is a regular function $\B \to \B$. Since $f(q) -f(q_0)$ has a zero at $q_0$, by theorem \ref{quotients} the product $\tilde f$ has the additional property that $\tilde f(q_0) = 0$. Inequalities \eqref{MainSchwarzPickEq} and \eqref{Req} now follow applying corollary \ref{schwarzp} to $\tilde f$ (taking into account that $\bar q_0 * q = q \bar q_0$). They are strict unless $\tilde f(q) = \mr_{q_0}(q) \cdot u$ for some $u \in \partial \B$, which is true if and only if $f$ is a regular M\"obius transformation of $\B$.

Inequality \eqref{culderiv} follows from corollary \ref{schwarzp} and from the fact that, according to formula \eqref{leibnizc}
$$\partial_c \tilde f(q) = \partial_c f(q) * (1-\overline{f(q_0)}*f(q))^{-*} + (f(q)-f(q_0)) *\partial_c (1-\overline{f(q_0)}*f(q))^{-*},$$
where the second term vanishes at $q_0$ by theorem \ref{quotients}.

As for \eqref{derivsfera}, it derives again from corollary \ref{schwarzp} proving that 
$$\partial_s \tilde f(q_0) =[1-\overline{f^s(q_0)}]^{-1} \partial_s f(q_0).$$ 
Indeed, setting $g(q) = (1-f(q)\overline{f(q_0)})^{-*}$, formula \eqref{leibnizs} and proposition \ref{transform} imply
\begin{align*}
\partial_s \tilde f(q_0) &= \partial_s\left[g(g)*(f(q)-f(q_0))\right]_{|_{q_0}} \\
&= \partial_s g(q_0) \cdot v_s (f(q)-f(q_0))_{|_{q_0}} + v_s g(q_0) \cdot \partial_s (f(q)-f(q_0))_{|_{q_0}}\\
&= -\partial_s g(q_0) \cdot Im(q_0) \cdot \partial_s f(q_0)+  v_s g(q_0) \cdot \partial_s f(q_0)\\
&=  \overline{ \left[ \overline{v_s g(q_0)} + Im(q_0) \overline{\partial_s g(q_0)}\right]}\cdot \partial_s f(q_0)\\
&= \overline{g^c(q_0)} \cdot \partial_s f(q_0)
\end{align*}
where we have taken into account that, according to \cite{perotti}, $v_s g^c(q) = \overline{v_s g(q)}$ and $\partial_s g^c(q) = \overline{\partial_s g(q)}$. Thanks to theorem \ref{quotients}, $g^c(q) = h^{-*}(q) = h(T_h(q))^{-1}$ where $h(q) = (1-f(q)\overline{f(q_0)})^c = 1-f(q_0)*f^c(q)$ and 
$$T_h(q) =  (1-f(q)\overline{f(q_0)})^{-1}\,  q\,  (1-f(q)\overline{f(q_0)}).$$
Since 
$$1-f(q_0)\overline{f(q_0)} = 1-|f(q_0)|^2$$ 
is real, $T_h(q_0) = q_0$ and $g^c(q_0) =h(q_0)^{-1}$. Furthermore, if $f(q) = \sum_{n \in \nn} q^n a_n$ then
$$f(q_0)*f^c(q)=f(q_0)*\sum_{n \in \nn} q^n \bar a_n =\sum_{n \in \nn} q^n f(q_0)\bar a_n = \sum_{m,n \in \nn} q^n q_0^m a_m\bar a_n$$
equals $f^s(q) = \sum_{k \in \nn} q^k \sum_{m = 0}^k a_m \bar a_{k-m}$ at $q_0$. Hence, $h(q_0) = 1-f^s(q_0)$, and the proof is complete.
\end{proof}

According to theorem \ref{quotients}
$$\partial_c f *(1-\overline{f(q_0)}*f(q))^{-*}_{|_{q_0}} =  \partial_c f(q_0) (1-\overline{f(q_0)}f(\tilde q_0))^{-1}$$
where $\tilde q_0 = T_g\left(\partial_c f(q_0)^{-1}q_0 \partial_c f(q_0)\right)$ and $g(q) = 1-\overline{f(q_0)}*f(q)$. Hence, inequality \eqref{culderiv} can be restated as
$$ \frac{|\partial_c f(q_0)|}{|1-\overline{f(q_0)}f(\tilde q_0)|} \leq  \frac{1}{1-|q_0|^2}$$
which closely resembles the complex estimate \eqref{complexdiff}.


\section{Applications of the Schwarz-Pick lemma}\label{sectionapplications}

As an application of our main result \ref{MainSchwarzPick}, we study the fixed point case extending the work done in \cite{rigidity}. In the proofs, we thoroughly use the following property of the zero set proven in \cite{zeros} (an immediate consequence of theorem \ref{quotients}).

\begin{corollary}\label{productzeros}
Let $f,g,h$ be regular functions on $B=B(0,R)$. Then $f*g$ vanishes at a point $q_0 = x_0 +I y_0$ if, and only if, either $f(q_0)=0$ or $f(q_0)^{-1}q_0 f(q_0) = x_0+ f(q_0)^{-1}If(q_0) y_0$ is a zero of $g$. As a consequence, if $f,h$ vanish nowhere in $B$  then each $x_0+y_0\s \subset B$ contains as many zeros of $f*g*h$ as zeros of $g$.
\end{corollary}

We are now ready for our study. The next result had been proven in \cite{rigidity} in the special case when $q_0$ is real, in the interval $(-1,1)$.

\begin{theorem}\label{fixedpoints}
Let $f : \B \to \B$ be a regular function and suppose $f$ to have a fixed point $q_0 \in \B$. Then either $f$ is the identity function, or $f$ has no other fixed point in $\B$.
\end{theorem}

\begin{proof}
Since $f(q_0) = q_0$, the inequality \eqref{MainSchwarzPickEq} becomes $|\tilde f| \leq |\mr_{q_0}|$ with
$$\tilde f(q) = (f(q)-q_0)*(1-\bar q_0*f(q))^{-*} = (1-f(q)\bar q_0)^{-*}*(f(q)-q_0).$$
Let us consider the set of points where $\tilde f$ and $\mr_{q_0}$ coincide: thanks to corollary \ref{productzeros}, each $x+y\s \subset \B$ contains as many zeros of $\tilde f - \mr_{q_0}$  as zeros of 
\begin{align*}
&(1-f(q)\overline{q_0})*(\tilde f(q) - \mr_{q_0}(q))*(1-q\bar q_0)=\\
& (f(q)-q_0)*(1-q\bar q_0) - (1-f(q)\bar q_0)*(q-q_0)=\\
& f(q) *[1-q\bar q_0 + \bar q_0*(q-q_0)] - q_0 *(1-q\bar q_0) - (q-q_0) =\\
& f(q) (1-|q_0|^2) - q (1-|q_0|^2) = [f(q)-q] (1-|q_0|^2).
\end{align*}
Clearly, the zero set of the last function is the fixed point set of $f$. 

Now let us suppose $f$ to have another fixed point $q_1=x_1+Iy_1\neq q_0$. If $q_0 \in x_1+y_1\s$ then the fixed point set contains the whole $2$-sphere $x_1+y_1\s$, and so does the zero set of $\tilde f - \mr_{q_0}$; in particular $|\tilde f(x_1+Jy_1)| = |\mr_{q_0}(x_1+Jy_1)|$ for all $J \in \s$. If, on the contrary, $q_0$ and $q_1$ lie in different spheres, then $\tilde f - \mr_{q_0}$ has a zero $\tilde q_1 \in x_1+y_1\s$ and in particular $|\tilde f(\tilde q_1)| = |\mr_{q_0}(\tilde q_1)|$ with $\tilde q_1 \neq q_0$. In both cases, according to the regular Schwarz-Pick lemma (theorem \ref{MainSchwarzPick}), $f$ must be a regular M\"obius transformation of $\B$.

We are left with proving that a regular M\"obius transformation of $\B$ having more than one fixed point in $\B$ must be the identity function. According to corollary \ref{productzeros}, for all $a \in \B, u \in \partial \B$, the difference $(1-q\bar a)^{-*}*(q-a) u -q$ has more than one zero in $\B$ if, and only if,
$$P(q) = (q-a) u-(1-q\bar a)*q = q^2 \bar a + q (u-1) -a u$$
does. Now, if the last polynomial factorizes as $P(q) = (q-\alpha)*(q-\beta) \bar a$ then $\alpha \beta\bar a =  -a u$. Either $a=0$ (in which case $P \equiv 0$ and the transformation coincides with the identity) or $|\alpha\beta| = 1$. In the latter case, $\alpha$ and $\beta$ cannot both lie in $\B$, and $P(q)$ cannot have more than one zero in $\B$. Thus, $(1-q\bar a)^{-*}*(q-a) u$ cannot have more than one fixed point in $\B$, and our proof is complete.
\end{proof}

As a byproduct of the previous proof, we observe that a regular M\"obius transformation of $\B$ having a fixed point in $\B$ either is the identity or has no other fixed point in $\overline{\B}$.

Another nice application of our main theorem \ref{MainSchwarzPick} is a direct proof of a result of \cite{rigidity}: the analog of Cartan's rigidity theorem for regular functions. In that paper, the result was proven using a ``slicewise'' technique: that is, reducing to the complex Cartan theorem. A direct approach is now possible, and it allows a generalization of the statement.

\begin{theorem}\label{cartan}
Let $f : \B \to \B$ be a regular function and suppose $f$ to have a fixed point $q_0 \in \B$. The following facts are equivalent:
\begin{enumerate}
\item[$(1)$] $f$ coincides with the identity function;
\item[$(2)$] the real differential of $f$ at $q_0$ is the identity;
\item[$(3)$] the Cullen derivative $\partial_c f (q_0)$ equals $1$;
\item[$(4)$] the spherical derivative $\partial_s f(q_0)$ equals $1$;
\item[$(5)$] $R_{q_0}f(q)$ equals $(1-\bar q_0 * q)^{-*}*(1-\bar q_0*f(q))$ at some $q \in \B$.
\end{enumerate}
\end{theorem}

The new proof is based on a technical lemma.

\begin{lemma}
Let $f : \B \to \B$ be a regular function and let $q_0 \in \B$. If $R_{q_0}f(q)$ equals the quotient 
\begin{equation}\label{extremalquotient}
(1-\bar q_0*q)^{-*}*(1-\overline{f(q_0)}*f(q))
\end{equation}
at any point of $\B$, then $f$ is a M\"obius transformation of $\B$.
\end{lemma}

\begin{proof}
By theorem \ref{MainSchwarzPick},
$$|R_{q_0}f(q)*(1-\overline{f(q_0)}*f(q))^{-*}| \leq |(1-\bar q_0*q)^{-*}|$$
and $f$ is a regular M\"obius transformation if equality holds at any point of $\B$. This is true, in particular, if 
$$R_{q_0}f(q)*(1-\overline{f(q_0)}*f(q))^{-*} - (1-\bar q_0*q)^{-*}$$
vanishes at any point of $\B$. This is equivalent to the vanishing of 
$$R_{q_0}f(q)-(1-\bar q_0*q)^{-*}*(1-\overline{f(q_0)}*f(q))$$ at some $q \in \B$.
\end{proof}

\begin{proof}[Proof of theorem \ref{cartan}]
$(1)\Rightarrow (2) \Rightarrow (3)$: these implications are obvious.

\noindent $(3) \Rightarrow (5)$: we already know that $R_{q_0}f(q_0) = \partial_c f (q_0)$. Moreover if $f(q_0)=q_0$ then the quotient 
$$Q(q)=(1-q\bar q_0)^{-*}*(1-\bar q_0*f(q))$$
equals $1$ at $q_0$, thanks to the fact that $\bar q_0*f(q) = \bar q_0f(\bar q_0^{-1}q\bar q_0)$ (by theorem \ref{quotients}).

\noindent $(5)\Rightarrow (1)$:
in the case of a fixed point $q_0$, quotient \eqref{extremalquotient} equals the aforementioned $Q(q)$. According to the previous lemma, if $Q(q)$ equals $R_{q_0}f(q)$ at any point of $\B$ then $f$ is a regular M\"obius transformation. We are left with proving that if $f$ is a regular M\"obius transformation $f(q) = (1-q \bar a)^{-*}*(q-a) u$ and if
\begin{align*}
&R_{q_0}f(q) -Q(q) =\\
&\left[(q-q_0)^{-*}+(1-q\bar q_0)^{-*}\bar q_0\right]*f(q) - (q-q_0)^{-*} q_0 - (1-q\bar q_0)^{-*}=\\
&(q-q_0)^{-*}*(1-q\bar q_0)^{-*}*\left\{\left[1-q\bar q_0 + (q-q_0)\bar q_0\right]*f(q) - (1-q\bar q_0)q_0 -(q-q_0)\right\}=\\
&(q-q_0)^{-*}*(1-q\bar q_0)^{-*}*(1-|q_0|^2)*[f(q)-q]
\end{align*}
has a zero in $\B$ then $f$ is identity.
The latter condition is equivalent to the existence of a zero $q_1\in \B$ for 
\begin{align*}
(1-q\bar q_0)*\left[R_{q_0}f(q) -Q(q)\right] (1-|q_0|^2)^{-1} =(q-q_0)^{-*}*[f(q)-q],
\end{align*}
i.e., to the existence of a regular $g : \B \to \hh$ such that
$$f(q) -q= (q-q_0)*(q- q_1)*g(q).$$
We have already seen in the proof of theorem \ref{fixedpoints} that $f(q)-q = (1-q \bar a)^{-*} * (q-\alpha)*(q-\beta) \bar a$ where $\alpha, \beta$ cannot both lie in $\B$. These two facts are only compatible if $\bar a=0$ and $g \equiv 0$, that is, if $f = id$.

\noindent$(1)\Leftrightarrow(4)$: when $f(q_0)=q_0$ then the spherical derivative 
$$\partial_s f(q_0) = (q_0 - \bar q_0)^{-1}(q_0 - f(\bar q_0))$$ equals $1$ if, and only if, $\bar q_0$ is a fixed point of $f$, too. According to theorem \ref{fixedpoints}, the latter is equivalent to $f=id$.
\end{proof}

For the sake of completeness, we conclude this section identifying explicitly which regular M\"obius transformations fix a point $q_0 \in \B$. Let us denote
$$S^3 = \left\{
\begin{bmatrix}
v & 0 \\
0 & 1
\end{bmatrix} 
: v \in \partial \B \right\} \leq Sp(1,1)$$

\begin{proposition}
For each $q_0 \in \B$, the class of regular M\"obius transformations fixing $q_0$ is the orbit of the identity function under the right action of the subgroup 
\begin{equation}
C({q_0})\cdot S^3 \cdot C({q_0})^{-1}\leq Sp(1,1)
\end{equation}
with $C({q_0})=\begin{bmatrix}
1 & -\bar q_0 \\
-q_0 & 1
\end{bmatrix}$. In other words, $(1-q\bar a)^{-*}*(q-a) u$ fixes $q_0$ if, and only if,
\begin{eqnarray}
u &=& (1-q_0v\bar q_0)^{-1}(v-|q_0|^2)\\
a &=& q_0 (1-\bar v) (1-q_0\bar v\bar q_0)^{-1}
\end{eqnarray}
for some $v \in \partial \B$.
\end{proposition}

\begin{proof}
If $f$ is a regular M\"obius transformation fixing $q_0$ then $\tilde f = f\,.\,C(q_0)$ is a regular M\"obius transformation mapping $q_0$ to $0$. Hence
$\tilde f(q) = (1-q\bar q_0)^{-*}*(q-q_0) v$
for some $v \in \partial \B$. In other words,
$$f\,.\,C(q_0) = id\,.\, C(q_0)\, . \begin{bmatrix}
v & 0 \\
0 & 1
\end{bmatrix}$$
or, equivalently, $f = id\,.\,A$ for some $A \in C({q_0})\cdot S^3 \cdot C({q_0})^{-1}$.
The final statement follows by direct computation, forcing
$$\begin{bmatrix}
1 & -\bar q_0 \\
-q_0 & 1
\end{bmatrix}
\begin{bmatrix}
v & 0 \\
0 & 1
\end{bmatrix}
\begin{bmatrix}
1 & -\bar q_0 \\
-q_0 & 1
\end{bmatrix}^{-1}$$
and
$$\begin{bmatrix}
1 & -\bar a \\
-a & 1
\end{bmatrix}
\begin{bmatrix}
u & 0 \\
0 & 1
\end{bmatrix}$$
to induce the same transformation.
\end{proof}


\section{Higher order estimates}\label{sectionhigher}

As in the complex case, the quaternionic Schwarz-Pick lemma admits higher order generalizations. Let us denote the $n$th Cullen derivative of a regular function $f$ as $\partial_c^nf$. Let $(q-q_0)^{*n} = (q-q_0)*\ldots*(q-q_0)$ denote the $*$-product of $n$ copies of $q \mapsto q-q_0$. The next theorem was proven in \cite{powerseries}.

\begin{theorem}
Let $\Omega$ be a domain in $\hh$. A function $f : \Omega \to \hh$ is regular if and only if, for each $q_0 \in \Omega$ 
\begin{equation}\label{regularseries}
f(q) = \sum_{n \in \nn} (q-q_0)^{*n} \frac{\partial_c^nf(q_0)}{n!}
\end{equation}
in a ball centered at $q_0$ with respect to the non-Euclidean distance
\begin{equation}
\sigma(q,p) = \left\{
\begin{array}{ll}
|q-p| & \mathrm{if\ } p,q \mathrm{\ lie\ on\ the\ same\ complex\ plane\ } L_I\\
\omega(q,p) &  \mathrm{otherwise}
\end{array}
\right.
\end{equation}
where
\begin{equation}
\omega(q,p) = \sqrt{\left[Re(q)-Re(p)\right]^2 + \left[|Im(q)| + |Im(p)|\right]^2}. 
\end{equation}
\end{theorem}

Theorem \ref{MainSchwarzPick} extends to the next result.

\begin{theorem}
Let $f : \B \to \B$ be a regular function, let $q_0 \in \B$. If $\partial_c^mf(q_0)=0$ for $1\leq m\leq n-1$ then
\begin{equation}
|(f(q)-f(q_0))*(1-\overline{f(q_0)}*f(q))^{-*}| \leq |(q-q_0)^{*n}*(1-\bar q_0*q)^{-*n}|
\end{equation}
for $q \in \B$. Furthermore,
\begin{equation}
\left|\partial_c^{n}f*(1-\overline{f(q_0)}*f)^{-*}\right|_{|_{q_0}} \leq \frac{n!}{(1-|q_0|^2)^{n}}
\end{equation} 
\end{theorem}

We point out that $\frac{n!}{(1-|q_0|^2)^{n}}$ is the $n$th Cullen derivative of 
$$(q-q_0)^{*n}*(1-\bar q_0*q)^{-*n} = \mr_{q_0}^{*n}(q).$$

\begin{proof}
If $\partial_c^mf(q_0)=0$ for $1\leq m\leq n-1$ then setting $q_1 = f(q_0)$ and $\tilde f = (f-q_1)*(1-\bar q_1*f)^{-*}$ defines a regular $\tilde f : \B \to \B$ with 
$$\partial_c^m\tilde f = \sum_{k=0}^{m-1} \partial_c^{m-k}f*\partial_c^k(1-\bar q_1*f)^{-*} \binom{m}{k} + (f-q_1)*\partial_c^m(1-\bar q_1*f)^{-*}.$$
Hence, $\partial_c^m \tilde f(q_0)=0$ for $0\leq m\leq n-1$, i.e., $\tilde f(q) = (q-q_0)^{*n} * g(q)$ for some regular $g : \B \to \hh$. Reasoning as in theorem \ref{teo1}, we prove that $\mr_{q_0}^{-*n}*\tilde f(q)$ is a regular function $\B \to \B$ and derive that for all $q \in \B$
$$|\tilde f(q)| \leq \left|\mr_{q_0}^{*n}(q)\right|$$
and 
$$|(q-q_0)^{-*n}*\tilde f(q)| \leq | (1-\bar q_0*q)^{-*n}|.$$
The latter implies
$$|\partial_c^{n}\tilde f(q_0) | \leq \frac{n!}{(1-|q_0|^2)^{n}}$$
and observing
$$\partial_c^{n}\tilde f(q_0) = \left[\partial_c^{n}f*(1-\bar q_1*f)^{-*}\right]_{|_{q_0}}$$
completes the proof.
\end{proof}

Finally, we generalize theorem \ref{MainSchwarzPick} in a different direction. We recall that the spherical derivative $\partial_sf$ is constant on each sphere $x_0+y_0\s$ and that it is not regular unless it is constant. Hence, iterated spherical derivation is meaningless. However, it makes sense to iterate the operator $R_{q_0}$ defined by formula \eqref{R}. This led in \cite{expansion} to the following result, where we use the notation
$$U(x_0+y_0\s,r) = \{q \in \hh \ |\  |(q-x_0)^2+y_0^2| < r^2\}$$
for all $x_0,y_0 \in \rr$, $r>0$, and we denote the composition of $R_{\bar q_0}$ and $R_{q_0}$ by juxtaposition and the $n$th iterate of $R_{\bar q_0}R_{q_0}$ by $(R_{\bar q_0}R_{q_0})^n$.

\begin{theorem}
Let $f$ be a regular function on $\Omega=B(0,R)$, and let  $U(x_0+y_0\s,r) \subseteq \Omega$. Then for each $q_0 \in x_0+y_0 \s$ there exists $\{A_n\}_{n \in \nn} \subset \hh$ such that
\begin{equation}\label{expansion}
f(q) = \sum_{n \in \nn} [(q-x_0)^2+y_0^2]^n [A_{2n}+(q-q_0) A_{2n+1}]
\end{equation}
for all $q \in U(x_0+y_0\s,r)$. 
Namely, $A_{2n} = (R_{\bar q_0}R_{q_0})^nf(q_0)$ and $A_{2n+1}= R_{q_0}(R_{\bar q_0}R_{q_0})^nf(\bar q_0)$ for all $n \in \nn$.
\end{theorem}

We are now ready for the announced higher order estimates. We recall that $g^s$ denotes the function obtained from $g$ by symmetrization (see definition \ref{conjugate}).

\begin{theorem}
Let $f : \B \to \B$ be a regular function and let $q_0 \in \B$. If the coefficients $A_m$ of the expansion \eqref{expansion} vanish for $1\leq m\leq 2n-1$ then
\begin{eqnarray}
|(f(q)-f(q_0))*(1-\overline{f(q_0)}*f(q))^{-*}| &\leq& |\mr_{q_0}^s(q)|^n\label{pari}\\
|(R_{\bar q_0}R_{q_0})^nf(q) *(1-\overline{f(q_0)}*f(q))^{-*}| &\leq& |(1-qx_0)^2+(qy_0)^2|^{-n}\nonumber
\end{eqnarray}
for $q \in \B$. If $A_{2n}=0$ as well then
\begin{eqnarray}
|(f(q)-f(q_0))*(1-\overline{f(q_0)}*f(q))^{-*}| &\leq& |\mr_{q_0}^s(q)|^n |\mr_{q_0}(q)| \label{dispari}\\
|R_{q_0}(R_{\bar q_0}R_{q_0})^nf(q) *(1-\overline{f(q_0)}*f(q))^{-*}| &\leq&|(1-qx_0)^2+(qy_0)^2|^{-n}\cdot|(1-q \bar q_0)^{-*}|\nonumber
\end{eqnarray}
\end{theorem}

\begin{proof}
Let us set $\tilde f (q)= (f(q)-f(q_0))*(1-\overline{f(q_0)}*f(q))^{-*}$. This defines a regular $\tilde f : \B \to \B$ with $\tilde f(q_0) = 0$ and
$$R_{q_0}\tilde f(q)=(q-q_0)^{-*}*\tilde f(q)=R_{q_0}f(q) *(1-\overline{f(q_0)}*f(q))^{-*}.$$
Thanks to the hypothesis $A_1 = R_{q_0}f(\bar q_0)=0$, we conclude that $R_{q_0}\tilde f(\bar q_0) =0$ so that 
\begin{align*}
R_{\bar q_0}R_{q_0}\tilde f(q)&=(q-\bar q_0)^{-*}*R_{q_0}\tilde f(q) = [(q-x_0)^2 + y_0^2]^{-1} \tilde f(q)\\
&=R_{\bar q_0}R_{q_0}f(q) *(1-\overline{f(q_0)}*f(q))^{-*}.
\end{align*}
Iterating this process, we conclude that
\begin{align*}
R_{q_0}(R_{\bar q_0}R_{q_0})^{k-1}\tilde f(q)&= [(q-x_0)^2 + y_0^2]^{-k+1} (q-q_0)^{-*} *\tilde f(q)\\
&=R_{q_0}(R_{\bar q_0}R_{q_0})^{k-1}f(q) *(1-\overline{f(q_0)}*f(q))^{-*},\\
(R_{\bar q_0}R_{q_0})^k\tilde f(q) &= [(q-x_0)^2 + y_0^2]^{-k} \tilde f(q)\\
&=(R_{\bar q_0}R_{q_0})^k f(q) *(1-\overline{f(q_0)}*f(q))^{-*}
\end{align*}
for all $1\leq k\leq n$, and that the coefficients $\widetilde A_m$ of the expansion of $\tilde f$ vanish for all $0\leq m\leq 2n-1$.
As a consequence, $\tilde f (q) = [(q-x_0)^2 + y_0^2]^n g(q)$ for some regular function $g$. Let us consider
$$\mr_{q_0}^s(q) = (1-q\bar q_0)^{-s} (q-q_0)^s = [(1-qx_0)^2+(qy_0)^2]^{-1} [(q-x_0)^2 + y_0^2]$$
and its $n$th power $(\mr_{q_0}^s)^n$: then
$$(\mr_{q_0}^s(q))^{-n} \tilde f(q) = [(1-qx_0)^2+(qy_0)^2]^ng(q)$$ 
is a regular function $h$ on $\B$. Now, $|h| =|\mr_{q_0}^s(q)|^{-n} |\tilde f(q)|\leq |\mr_{q_0}^s(q)|^{-n}$ where $\mr_{q_0}^s = \mr_{q_0}*\mr_{q_0}^c =  \mr_{q_0}*\mr_{\bar q_0}$ maps $\partial \B$ to $\partial\B$ by lemma \ref{modulusproduct}. Reasoning as in theorem \ref{teo1}, we can prove that $|h|\leq 1$ and equations \eqref{pari} follow by direct computation.

Finally, if $A_{2n}=0$ then $\widetilde A_{2n} = 0$ and $g(q) = [(q-x_0)^2 + y_0^2]^{-n} \tilde f(q)$ has a zero at $q_0$. Thus, $h(q) = [(1-qx_0)^2+(qy_0)^2]^ng(q)$ is a regular function $\B \to \B$ having a zero at $q_0$. By theorem \ref{MainSchwarzPick}, $|h| \leq |\mr_{q_0}|$ and equations \eqref{dispari} follow (making use of lemma \ref{modulusproduct}).
\end{proof}

We believe that the two results proven in this section, however technical their statements may appear, show that the quaternionic Schwarz-Pick lemma establishes a strong link between the differential and multiplicative properties of the regular self-maps of $\B$. This recalls the complex setting, but in a many-sided way that reflects the richness of the non-commutative context.


\bibliography{SchwarzPick}

\bibliographystyle{abbrv}


\end{document}